\documentclass[12pt,a4paper]{amsart}
\usepackage{}
\usepackage[all]{xy}
\usepackage{amsmath}
\usepackage{amsfonts}
\usepackage{amssymb}
\usepackage[all]{xy}           
\usepackage{mathrsfs}

\usepackage{bbding}
\usepackage{amscd}

\usepackage{ifpdf}
\ifpdf
 \usepackage[colorlinks,final,backref=page,hyperindex]{hyperref}
\else
  \usepackage[colorlinks,final,backref=page,hyperindex,hypertex]{hyperref}
\fi

\usepackage{tikz}
\usepackage[active]{srcltx}
\usepackage[dvips]{lscape}
\xyoption{web}
\usetikzlibrary{arrows}
\usepackage{xspace,amsthm}
\usepackage[parfill]{parskip}
\usepackage{mathrsfs}
\makeatletter

\def\C{\mathbb{C}}

\usepackage[active]{srcltx}
\usepackage{enumerate,amsthm,amssymb,latexsym}
\usepackage[all]{xy}
\usepackage[parfill]{parskip}

\usepackage{amsmath,amssymb,xspace,amsthm}
\newtheorem{theorem}{Theorem}
\newtheorem{remark}[theorem]{Remark}
\newtheorem{lemma}[theorem]{Lemma}
\newtheorem{proposition}[theorem]{Proposition}
\numberwithin{equation}{section}

\usepackage{multirow}
\usepackage{marginnote}

\def\span{\operatorname{span}}

\renewcommand{\l}{\ensuremath{\lambda}}

\newcommand{\Z}{\ensuremath{\mathbb{Z}}\xspace}

\newcommand{\supp}{\ensuremath{\operatorname{Supp}}\xspace}

\newcommand{\ad}{\operatorname{ad}\xspace}

\renewcommand{\phi}{\varphi}

\renewcommand{\leq}{\leqslant}
\renewcommand{\geq}{\geqslant}
\def\md{\mathfrak{D}}
\def\ms{\mathfrak{s}}
\def\mg{\mathfrak{g}}
\def\mh{\mathfrak{h}}

\def\sl{\mathfrak{sl}}
\def\so{\mathfrak{so}}
\def\gl{\mathfrak{gl}}

\def\l{\lambda}

\def\ad{\text{ad}}
\def\span{\text{span}}

\def\C{\Bbb C}
\def\Z{\Bbb Z}

\def\z{\dot z}

\begin{document}
\title [The Schr{\"o}dinger Lie algebra  in $(n+1)$-dimensional space-time]
{Irreducible weight modules over the Schr{\"o}dinger Lie algebra  in $(n+1)$ dimensional space-time}
\author{Genqiang Liu, Yang Li, Keke Wang}
\maketitle

\begin{abstract} In this paper, we study weight representations over the Schr{\"o}dinger Lie algebra $\mathfrak{s}_n$ for any positive integer $n$.
It turns out that the algebra $\mathfrak{s}_n$ can be realized by polynomial differential operators. Using this realization, we give a complete classification of
irreducible weight  $\mathfrak{s}_n$-modules with finite dimensional weight spaces for any $n$. All such modules can be clearly characterized by the tensor product of $\mathfrak{so}_n$-modules, $\mathfrak{sl}_2$-modules and modules over the Weyl algebra.
\end{abstract}


 \noindent {\em 2010 Math. Subj. Class.:}
17B10, 17B37, 17B81, 20G42.

\section{Introduction}

The Schr{\"o}dinger group  is the symmetry group of the free particle Schr{\"o}dinger equation, see \cite{DDM1,Pe}.
Its corresponding Lie algebra is the Schr{\"o}dinger algebra which plays an important role in mathematical physics
and its applications.  It is not semi-simple.
Recently there were a series of papers on studying
the representation theory of the Schr{\"o}dinger algebra
$\ms_1$ in the case of $(1+1)$-dimensional space-time, see \cite{AD,CCS,DDM1,D,DLMZ,WZ,ZC}.
In particular, all simple weight modules over $\ms_1$ with  finite dimensional weight spaces were
classified in \cite{D}, see also \cite{LMZ}.    All simple weight modules with infinite dimensional weight spaces over $\ms_1$ were
classified in \cite{BL1,BL2}.  The BGG category
$\mathcal{O}$ of $\ms_1$  was studied in \cite{DLMZ}. In the present paper, we will study
 weight  modules with finite dimensional weight spaces over $\ms_n$ in $(n+1)$-dimensional
  space-time,  for any positive integer $n$.

The paper is organized as follows. In Section 2, we give the related definitions and notations.
In Section 3, we firstly study the structure of the universal enveloping algebra $U(\ms_n)$
using differential operators, see Proposition \ref{proposition3}.
Then we give a  crystal clear description of all highest weight modules, see Proposition \ref{GHM}. The classification of Harish-Chandra modules was divided into two situations: $\z=0$ and $\z\neq 0$, according to the action
of the central element $z$. Let $M$ be an irreducible  Harish-Chandra
$\ms_n$-module on which  $z$ acts as the  scalar $\z$.
If $\z=0$, then $\mh_nM=0$ and  $M\cong V\otimes L$ for some
 finite dimensional irreducible $\so_n$-module $V$ and some  irreducible weight $\sl_2$-module $L$, see Theorem \ref{p14}.
If $\z\neq 0$ and $n>1$, then the module $M$ is a highest weight module or a lowest weight module,  isomorphic or equivalent  to $  V\otimes_{\z} L_{\sl_2}(\lambda) \otimes _{\dot z} \C[t_1,\dots,t_n]$ up to some involution  of $\ms_n$ for some finite dimensional  irreducible
$\so_{n}$-module $V$ and some irreducible highest weight module $L_{\sl_2}(\lambda)$ over $\sl_2$, see (a), (b) in Theorem \ref{th9}.  In the case that $n=1$, in addition to highest and lowest weight modules, there is a third class of modules $L_{\sl_2}(k)\otimes_{\z} t_1^\lambda\C[t_1^{\pm 1}]$, $k\in\Z_+,\lambda\not\in\Z$, see (c) in Theorem \ref{th9}. It should be noted that the classification for $n=1$ was also given in \cite{D} by a different method.

Throughout  this paper, we denote by $\mathbb{Z}$, $\mathbb{Z}_+$, $\mathbb{N}$,
$\mathbb{C}$ and $\mathbb{C}^*$ the sets of  all integers, nonnegative integers,
positive integers, complex numbers, and nonzero complex numbers, respectively.
For any Lie algebra $\mg$, we denote its universal enveloping algebra by $U(\mg)$.


\section{Definitions and notations}

\subsection{The Schr{\"o}dinger algebra  $\ms_n$ in $n+1$-dimensional space-time}

Fix a positive integer $n$.
In this subsection, we first recall the definition of $\ms_n$ from \cite{DDM1} in a different form. 
We know that  the general linear Lie algebra $\mathfrak{gl}_{2n}$  has  the natural representation on $\C^{2n}$ by
left matrix  multiplication. Let $\{e_1, e_2, \cdots, e_{2n}\}$ be  the standard basis of $\C^{2n}$.

The Heisenberg Lie algebra $\mh_{n}=\C^{2n}\oplus \C z$ is the Lie algebra with
Lie bracket   given by
$$[e_i, e_{n+i}]=z, \qquad [z, \mh_n]=0.$$

Recall that the Schr{\"o}dinger Lie   algebra $\ms_n$ is  the  semidirect product Lie algebra
$$\ms_n=(\sl_2\oplus \so_n )\ltimes \mh_n,$$
where $\sl_2$ is embedded in $\gl_{2n}$
by the mapping
$$\left(
    \begin{array}{cc}
      a & b \\
      c & -a \\
    \end{array}
  \right)\mapsto \left(
    \begin{array}{cc}
      a I_n& bI_n \\
      cI_n & -aI_n \\
    \end{array}
  \right)
$$
and $\so_n$ is embedded in $\gl_{2n}$ by
$$A\in \so_n\mapsto \left(
    \begin{array}{cc}
     A& 0 \\
      0 & A \\
    \end{array}
  \right). $$

  Here $I_n$ is the $n\times n$ identity matrix, $\sl_2\oplus \so_n$ acts
   on $\mh_n$ by matrix multiplication, and $[z,\ms_n]=0$ .

Next, we will introduce a basis of $\ms_n$. Let $$\aligned h&=\left(
    \begin{array}{cc}
       I_n& \\
       & -I_n \\
    \end{array}
  \right),
 e=\left(
    \begin{array}{cc}
      0& I_n \\
      0& 0\\
    \end{array}
  \right),
  f=\left(
    \begin{array}{cc}
      0& 0\\
      I_n & 0\\
    \end{array}
  \right),\\
s_{ij}&=\left(
    \begin{array}{cc}
     e_{ij}-e_{ji}& 0 \\
      0 &  e_{ij}-e_{ji} \\
    \end{array}  \right),1\leq i<j\leq n,\\
    x_k&=e_k, y_k=e_{n+k}, 1\leq k\leq n, \endaligned $$where $e_{i,j}(1\leq i,j\leq n)$ the $n\times n$ matrix with zeros everywhere except a $1$ on position $(i,j)$.

Note that $s_{ij}=-s_{ji}$. We can check that the  non-trivial commutation relations of $\ms_n$ are
$$\aligned & [h,e]=2e, [h,f]=-2f, [e,f]=h,\\
& [x_i, y_i]=z, [h,x_i]=x_i, [h,y_i]=-y_i,\\
& [e,y_i]=x_i, [f,x_i]=y_i,\\
&[s_{kl},  x_i]=\delta_{li} x_k-\delta_{ki}x_l, [s_{kl},  y_i]=\delta_{li} y_k-\delta_{ki}y_l, \\
& [s_{ij},s_{kl}]=\delta_{kj}s_{il}+\delta_{il}s_{jk}+\delta_{lj}s_{ki}+\delta_{ki}s_{lj}.
\endaligned$$

Let $$\aligned \ms_n^+& =\span_\C \{ e,  x_i\mid i=1,\dots,n\}, \\
\ms_n^-&=\span_\C \{ f, y_i\mid i=1,\dots,n\},\\
\ms_n^0&=\so_n\oplus \C h\oplus \C z.\endaligned$$

Then we have the following triangular decomposition: $$\ms_n=\ms_n^-\oplus \ms_n^0\oplus \ms_n^+.$$

\subsection{Weight modules}
 An $\ms_n$-module  $M$ is called a {\it weight module} if $h$ acts diagonally on  $M$, i.e.
$$ M=\oplus_{\lambda\in \C } M_\lambda,$$
where $M_\lambda:=\{v\in M \mid hv=\lambda v\}.$   Denote $$\mathrm{Supp}(M):=\{\lambda\in\C  \mid M_\lambda\neq0\}.$$

A weight module is called a {\it Harish-Chandra module} if all its weight spaces are finite dimensional.
For a weight module $M$, a
weight vector $v\in M_\lambda$ is called a {\it highest weight vector} if $\ms_n^{+}v=0$.
 A module is called a {\it highest weight module}
if it is generated by a highest weight vector.

\subsection{Verma modules}
Denote $\mathfrak{b}_n:=\ms_n^0\oplus \ms_n^+$.
For $\lambda\in \C$ , $\z\in\C$ and an irreducible $\so_n$-module $V$,  we can make $V$ to be a
$\mathfrak{b}_n$-module such that
$$ \ms_n^+v=0, hv=\lambda v, z v=\z v, $$
for any $v\in V$. We denote this $\mathfrak{b}_n$-module by $V_{\lambda,\z}$

The {\em Verma module} is defined, as usual, as follows:
\begin{displaymath}
M(V, \lambda,\z):=\mathrm{Ind}_{\mathfrak{b}_n}^{\ms_n}V_{\lambda, \z}\cong
U(\ms_n)\bigotimes_{U(\mathfrak{b}_n)}V_{\lambda, \z}.
\end{displaymath}

The module $M(V, \lambda,\z)$ has the  unique simple quotient module denoted by
$ L(V, \lambda,\z)$. If $\dim V=1$, then we denote $M(V, \lambda,\z)$
 (resp. $L(V, \lambda,\z)$) by  $M(\lambda,\z)$ (resp. $L(\lambda,\z)$).
Similarly we can define the Verma module ${M}_{\sl_{2}}(\lambda)$ and
the irreducible highest weight module ${L}_{\sl_{2}}(\lambda)$ over $\sl_2$  for  $\lambda\in\C$ .

\section{Irreducible weight $\ms_n$-modules}

\subsection{Differential operator realization of $\ms_n$}

 For $n\in\mathbb{N}$,  denote by $\md_{n}$ the {\it Weyl algebra of rank} $n$ over
 the  polynomial algebra $\C[t_1,\cdots,t_n]$. Namely,
$\md_{n}$  is the associative algebra  over $\C$ generated by
$$t_1,\dots,t_n,
\partial_1,\dots,\partial_n$$ subject to
the relations
$$[\partial_i, \partial_j]=[t_i,t_j]=0,\qquad [\partial_i,t_j]=\delta_{i,j},\ 1\leq i,j\leq n.$$
A $\md_n$-module $V$ is a {\it weight module} if all $t_i\partial_i$ are semisimple on $V$.

In the next lemma, we will show that $\md_{n}$ is a quotient algebra of $U(\ms_n)$.

\begin{lemma}\label{lemma3}For any nonzero scalar $\z$, we have the
    associative algebra homomorphism
    $$ \theta_{\z}: \,\, U(\ms_n)/\langle z-\z\rangle\rightarrow  \md_n$$
    defined by
$$\aligned
          & x_i \mapsto  \tilde{x}_i:=\sqrt{\z}\partial_i,\\
          & y_i\mapsto \tilde{y}_i :=\sqrt{\z} t_i,\\
          & e\mapsto \tilde{e}:=\frac{1}{2}  \sum_{k=1}^n \partial_k^2,\\
          & f\mapsto \tilde{f}:=-\frac{1}{2} \sum_{k=1}^n t_k^2,\\
          &h\mapsto \tilde{h}:=-\frac{1}{2} \sum_{k=1}^n (\partial_kt_k+t_k\partial_k),\\
          & s_{ij}\mapsto \tilde{s}_{ij}:= t_i\partial_j-t_j\partial_i,
           \endaligned
$$ where $1\leq i, j \leq n$.
\end{lemma}

\begin{proof}
For any $i\in\{1,\dots,n\}$, we have that $[\tilde{x}_i, \tilde{y}_i]=\z,$

 $$\aligned  &  [\tilde{e}, \tilde{y}_i]=[\frac{1}{2} \sum_{k=1}^n \partial_k^2, \sqrt{\z} t_i]\\
 &=[\frac{1}{2} \partial_i^2, \sqrt{\z} t_i]=\sqrt{\z}\partial_i=\tilde{x}_i,\endaligned$$

 and $$\aligned & [\tilde{f},\tilde{x}_i]=[-\frac{1}{2} \sum_{k=1}^n t_k^2,\sqrt{\z}\partial_i]\\
  &=[-\frac{1}{2}  t_i^2,\sqrt{\z}\partial_i]=\sqrt{\z}t_i=\tilde{y}_i.\endaligned$$

 Furthermore, we can obtain that
  $$\aligned \ [\tilde{e},\tilde{f}]& =[\frac{1}{2}  \sum_{k=1}^n \partial_k^2,-\frac{1}{2} \sum_{k=1}^n t_k^2]
 =\frac{-1}{4}  \sum_{k=1}^n [\partial_k^2, t_k^2]\\
 &=\frac{-1}{2} \sum_{k=1}^n (\partial_kt_k+t_k\partial_k)=\tilde{h},\\
 \ [\tilde{h}, \tilde{e}]&=[-\frac{1}{2} \sum_{k=1}^n (\partial_kt_k+t_k\partial_k),\frac{1}{2}  \sum_{k=1}^n \partial_k^2]\\
 & =-\sum_{k=1}^n\frac{1}{4}[(\partial_kt_k+t_k\partial_k), \partial_k^2]\\
 &=\sum_{k=1}^n\partial_k^2=2\tilde{e},\endaligned$$

and   $$\aligned \ [\tilde{h}, \tilde{f}]&=[-\frac{1}{2} \sum_{k=1}^n (\partial_kt_k+t_k\partial_k),-\frac{1}{2}  \sum_{k=1}^n t_k^2]\\
 & =\sum_{k=1}^n\frac{1}{4}[(\partial_kt_k+t_k\partial_k), t_k^2]\\
 &=\sum_{k=1}^nt_k^2=-2\tilde{f}.\endaligned$$

 Finally, for any $k,l,i,j\in\{1,\dots,n\}$, we get that

 $$\aligned \ [\tilde{s}_{kl}, \tilde{x}_i]&=[t_k\partial_l-t_l\partial_k,\sqrt{\z}\partial_i]\\
 &= -\sqrt{\z}\delta_{ki}\partial_l+\sqrt{\z}\delta_{li}\partial_k\\
 &= -\delta_{ki}\tilde{x}_l+\delta_{li}\tilde{x}_k,\endaligned$$

 $$\aligned \ [\tilde{s}_{kl}, \tilde{y}_i]&=[t_k\partial_l-t_l\partial_k,\sqrt{\z}t_i]\\
 &= -\sqrt{\z}\delta_{ki}t_l+\sqrt{\z}\delta_{li}t_k\\
 &= -\delta_{ki}\tilde{y}_l+\delta_{li}\tilde{y}_k,\endaligned$$

 and $$\aligned \ [\tilde{s}_{ij},\tilde{s}_{kl}]
 =\delta_{kj}\tilde{s}_{il}+\delta_{il}\tilde{s}_{jk}+\delta_{lj}\tilde{s}_{ki}+\delta_{ki}\tilde{s}_{lj}.\endaligned$$

 Therefore,  the map $\theta_{\z}$ can define an algebra homomorphism.
\end{proof}

\begin{remark} Via the homomorphism $\theta_{\z}$ in Lemma \ref{lemma3}, any $\md_n$-module $M$
can be viewed as an $\ms_n$-module by $x \cdot v=\theta_{\z}(x)v$, for $ x\in\ms_n, v\in M$.  In particular, the
module $\C[t_1,\cdots,t_n]$ is isomorphic to the irreducible highest weight module $L(-\frac{n}{2}, \z)$
\end{remark}
Using Lemma \ref{lemma3}, we have the following algebra isomorphism.

 \begin{proposition}\label{proposition3}For any nonzero scalar $\z$, we have the
    associative algebra isomorphism
    $$ \phi_{\z}: \,\, U(\ms_n)/\langle z-\z\rangle\rightarrow U(\so_n\oplus \sl_2)\otimes \md_n$$
    defined by
$$\aligned
          & x_i \mapsto 1\otimes \tilde{x}_i,\\
          & y_i\mapsto1\otimes \tilde{y}_i,\\
          & e\mapsto e\otimes 1+1\otimes  \tilde{e},\\
          & f\mapsto f\otimes 1+1\otimes  \tilde{f},\\
          &h\mapsto h\otimes 1+1\otimes \tilde{h},\\
          & s_{ij}\mapsto s_{ij}\otimes 1+ 1\otimes \tilde{s}_{ij},
           \endaligned
$$ where $1\leq i, j \leq n$.
 \end{proposition}

\begin{proof} Let $\pi$ denote the canonical homomorphism
$$U(\ms_n)/\langle z-\z\rangle \rightarrow U(\so_n\oplus \sl_2)=U(\ms_n)/\langle \mh_n\rangle.$$
Moreover, we use $\Delta$ to denote  the usual co-multiplication
$$U(\ms_n)/\langle z-\z\rangle\rightarrow U(\ms_n)/\langle z-\z\rangle\otimes U(\ms_n)/\langle z-\z\rangle,$$
defined by $$ \Delta(u)=u\otimes 1+ 1\otimes u,$$ for any $u\in U(\ms_n)/\langle z-\z\rangle$.
Since  $\phi_{\z}$ is equal to the composition of $\pi\otimes \theta_{\z}$ and $\Delta$,
the map $\phi_{\z}$ is an algebra homomorphism.

For $r_{ij}, m, q, l,r_k, s_d\in\Z_+$, we have  $$ \phi_{\z}(s_{ij}^{r_{i,j}}f^mh^{q}e^ly_k^{r_k}x_{d}^{s_d})
=s_{ij}^{r_{i,j}}f^mh^{q}e^l\otimes \tilde{y}_k^{r_k}\tilde{x}_{d}^{s_d} +\cdots+1\otimes \tilde{s}_{ij}^{r_{i,j}}\tilde{f}^m\tilde{h}^{q}\tilde{e}^l\tilde{y}_k^{r_k}\tilde{x}_{d}^{s_d}.$$
Then using the  lexicographic order of the  PBW basis, $\phi_{\z}$ maps
linearly independent elements to linearly independent elements, and hence $\phi_{\z}$ is injective. From
$\phi_{\z}(U(\mh_n)/\langle z-\z\rangle)=1\otimes \md_n$ and $ x\otimes 1=\phi_{\z}(x)-1\otimes \tilde{x}\in Im(\phi_{\z})$ for
 any $x\in\{e,h,f,s_{ij}\mid i,j=1,\dots,n\}$,  we can see that $\phi_{\z}$ is surjective. Therefore  $\phi_{\z}$ is an
isomorphism.
\end{proof}

\begin{remark} We know that 
$U(\so_n\oplus \sl_2)\cong U(\so_n)\otimes U(\sl_2)$.  For any $\so_n$-module $V$, $\sl_2$-module $N$ and $\md_n$-module $L$, through the algebra isomorphism $\phi_{\z}$,
the tensor product $V\otimes N \otimes L$ becomes an $\ms_n$-module, which is denoted by  $V\otimes_{\z} N \otimes_{\z} L$.
\end{remark}
\subsection{The structure of highest weight modules  over $\ms_n$}

In this subsection, we will use the isomorphism $\phi_{\z}$ in Proposition \ref{proposition3} to discuss highest weight modules  over $\ms_n$.

\begin{proposition}\label{GHM}
Let $V$ be an irreducible $\so_n$-module, $\lambda\in\C, \dot z\in\C$.
 \begin{enumerate}[$($a$)$]
\item If $\z=0$, then $\mh_nL(V, \lambda,\z)=0$.
\item If $\z\neq 0$, then we have that
\begin{equation}\label{3.1}M(V,\lambda,\z)\cong  V \otimes_{\z} {M}_{\sl_{2}}(\lambda+\frac{n}{2})\otimes_{\z} \C[t_1,\dots,t_n], \end{equation}
and \begin{equation}\label{3.2} L(V, \lambda,\z)\cong V\otimes_{\z}  {L}_{\sl_2}(\lambda+\frac{n}{2})\otimes _{\z}\C[t_1,\dots,t_n].\end{equation}
\item If $\z\neq 0$, then the Verma module $M(V,\lambda,\z)$ is irreducible if and only if $\lambda+\frac{n}{2}\not\in \Z_+ $.
\end{enumerate}
\end{proposition}

\begin{proof}
(a) As $\z=0$,  we have that $$x_i y_j (1\otimes V)=y_j x_i(1\otimes V)=0.$$ Then from $$e y_j(1\otimes V)=(x_j+y_je)(1\otimes V)=0,$$ we see that
$$U(\ms_n)\mh_n(1\otimes V)=U(\ms_n)\sum_{j=1}^ny_j(1\otimes V)=U(\ms_n^-)\sum_{j=1}^ny_j(1\otimes V)$$
which is a proper submodule of $M(V,\lambda,\z)$.
Since  the module $L(V, \lambda,\z)$ is the unique irreducible quotient module of  $M(V,\lambda,\z)$,
 we can see  that $\mh_nL(V, \lambda,\z)=0$.

(b)  In the module $ V \otimes_{\z} {M}_{\sl_{2}}(\lambda+\frac{n}{2})\otimes_{\z} \C[t_1,\dots,t_n]$,
we can see that  $$\aligned  h(v\otimes w_{\lambda+\frac{n}{2}}\otimes 1)&=\lambda(v\otimes w_{\lambda+\frac{n}{2}}\otimes 1),\\
 \ms_n^+(v\otimes w_{\lambda+\frac{n}{2}}\otimes 1)&=0,\endaligned $$ 
 where $v\in V$, and $w_{\lambda+\frac{n}{2}}$ is
the highest weight vector of the $\sl_2$-Verma module ${M}_{\sl_{2}}(\lambda+\frac{n}{2})$.

By the universal property of Verma modules, the map
$$1\otimes v\mapsto v\otimes w_{\lambda+\frac{n}{2}}\otimes 1,  v\in V, $$ induces an
$\ms_n$-module homomorphism
 $$\psi: M(V,\lambda,\z)\rightarrow V \otimes_{\z} {M}_{\sl_{2}}(\lambda+\frac{n}{2})\otimes_{\z} \C[t_1,\dots,t_n]$$
which is surjective. For $k,r_1,\cdots,r_n\in\Z_+$,
$$\psi(y_1^{r_1}\cdots y_n^{r_n}f^k(1\otimes v))= \sum_{i=0}^k v\otimes f^i v_{\lambda+\frac{n}{2}}\otimes g_i(t_1,\dots,t_n),$$
where $g_i(t_1,\dots,t_n)$ is a homogenous polynomial of degree $2(k-i)+r_1+\dots+ r_n$.
By PBW theorem, $\psi$ maps linearly independent elements to linearly independent elements, and hence $\psi$ is injective.
Therefore,  $\psi$ is an isomorphism. Then the isomorphisms (\ref{3.1}) and (\ref{3.2}) follow.

(c) By the isomorphism (\ref{3.1}) and the simplicity
 of the $\so_n$-module $V$ and the $\md_n$-module $\C[t_1,\dots,t_n]$, the $\ms_n$-module $M(V,\lambda,\z)$ is irreducible
 if and only if ${M}_{\sl_{2}}(\lambda+\frac{n}{2})$ is irreducible. Then (c) follows.
\end{proof}

\subsection{Classification of irreducible Harish-Chandra modules}

In this subsection, we will classify irreducible Harish-Chandra modules over $\ms_n$. Firstly we will introduce two useful lemmas.

\begin{lemma}\label{lemma6} Let $M$ be an irreducible $\ms_n$-module. For any $x\in \{e,f,x_i,y_i \mid i=1,\dots,n\}$,  we have that
$x$  acts either  injectively on $M$ or locally nilpotently on $M$.
\end{lemma}

Lemma \ref{lemma6} follows from the fact that $\ad x$ acts locally nilpotently on $U(\ms_n)$, we can refer to \cite{DMP}.

\begin{lemma}  \label{lemma-nil}Let $M$ be any irreducible Harish-Chandra
$\ms_n$-module, on which  $z$ acts as some $\z\in\C$.
 If $e$ (resp. $f$) acts locally nilpotently on $M$, then so does $x_i$ (resp. $y_i$) for any $i\in\{1,\dots,n\}$.
If $\z\neq 0$,  then  the converse is also true.
 \end{lemma}

\begin{proof} Suppose  that $e$ acts locally nilpotently on $M$.
Since all weight  spaces of $M$ are finite dimensional, by the theory of highest weight
modules over $\sl_2$,
for any
$\lambda \in \supp(M)$, there exists a $k\in \Z_+$ such that $\lambda+2l\not \in \supp{M}$ for any $l>k$.
Then from $x_i M_\lambda \subset M_{\lambda+1}$, we have that  $x_i$ acts nilpotently on $M_\lambda$ for
any $i\in\{1,\dots,n\}$.
Then by Lemma $\ref{lemma6}$, $x_i$ acts locally nilpotently on $M$.
Conversely, if $x_i$ acts locally nilpotently on
$M$ and $\z\neq 0$, then similarly by the theory of highest weight
modules over the Heisenberg algebra $\C x_i+\C y_i+\C z$, we can also have that
$e$ acts locally nilpotently on $M$.
\end{proof}

Next we will recall the twisted localization functor from \cite{M}.
Since $\ad f$ acts locally nilpotent on $U(\ms_n)$,  the multiplicative subset $$R= \{ f^i \mid i\in\Z_+\}$$
  is an Ore subset of $U(\ms_n)$,  and hence we have the corresponding Ore
localization $U(\ms_n)_{f}$ with respect to the subset $R$,  see \cite{M}.

For $b\in \C$, there is an isomorphism $\gamma_b$ of $U(\ms_n)_f$ such that
$$\gamma_b(u)=\sum_{j\geq 0}\binom{b}
{j} (\text{ad} f)^{j} (u) f^{-j},\ u\in U(\ms_n)_f .$$

Note that for $b=k\in \Z$, we have $\gamma_b(u)=f^kuf^{-k}$ for any $u\in U(\ms_n)_f $, see \cite{M}.
Next we can check that
\begin{equation}\label{Iso1}   \gamma_b(e)=e+b(1-b-h)f^{-1},
 \end{equation}
\begin{equation}\label{Iso} \gamma_b(x_i)=x_i+by_if^{-1},\end{equation}
for any $1\leq i\leq n$.

For a $U(\ms_n)_f $-module $M$, it can be twisted by $\gamma_b$  to be a new
 $U(\ms_n)_f$-module $M^{\gamma_b}$.  As vector spaces $M^{\gamma_b}= M$.
 For $v\in M^{\gamma_b}$, $x\in \ms_n$,
 $x\cdot v=\gamma_b(x)v $. The modules $M$ and  $M^{\gamma_b}$ are said to be equivalent.

\begin{theorem}\label{p14}
Let $M$ be an irreducible  Harish-Chandra
$\ms_n$-module, on which  $z$ acts as   zero.
 Then $\mh_nM=0$, i.e., $M$ is an irreducible $\so_n\oplus \sl_2$-module. Hence, $M\cong V\otimes L$ for some
 finite dimensional irreducible $\so_n$-module $V$ and some  irreducible weight $\sl_2$-module $L$.
\end{theorem}

\begin{proof} If  $e$ acts locally nilpotently on $M$, then by Lemma \ref{lemma-nil}, $M$ is an irreducible
highest weight module. By (a) in Proposition \ref{GHM}, $\mh_n M=0$. In the case that $f$ acts locally nilpotently on $M$,
we also have that $\mh_n M=0$ by the similar arguments.

Next, we suppose that both $e$ and $f$ act injectively on $M$.
Since all weight  spaces of $M$ are finite dimensional, both $e$ and $f$ acts bijectively on $M$.
Then $M$ can be naturally  viewed as a $U(\ms_n)_f$-module. Since $M^{\gamma_b}$ is a uniformly
bounded weight $\ms_n$-module,  $M^{\gamma_b}$ has an irreducible $\ms_n$-submodule $N$. Let $\l\in\supp{N}$.
 Since $\dim N_\lambda <\infty$, there
is a $v\in N_\lambda$ and $a\in\C^*$ such that $ef v=av.$ Choose  $b\in\C$ such that
$a+b(1-b-\lambda)=0$. Then
$$\gamma_b(e)fv=efv+b(1-b-h)v=0.$$
By Lemma \ref{lemma6}, $\gamma_b(e)$ acts locally nilpotently on $N$.
Then $N$ is an irreducible highest weight module. Consequently
$\gamma_b(\mh_n) N=0$. From
$$ \gamma_b(y_i)=y_i, \gamma_b(x_i)=x_i+by_if^{-1}, 1\leq i\leq n,$$
we see that $\mh_n N=0$. So $K=\{ v\in M\mid \mh_n v=0\}$ is a nonzero $\ms_n$-submodule
of $M$. The simplicity of $M$ tells us that $K=M$, i.e., $\mh_n M=0$. 
So $M$ is an irreducible  $\so_n\oplus \sl_2$-module.  Since
$U(\so_n\oplus \sl_2)\cong U(\so_n)\otimes U(\sl_2)$,  there are an irreducible
$\so_n$-module $V$ and an  irreducible weight $\sl_2$-module $L$ such that $M\cong V\otimes L$.
 The condition that each weight space  (with respect to the action of $h$) of $M$ is finite dimensional  forces that
 $\dim V<\infty$.
 The proof is complete.
\end{proof}

\begin{lemma}\label{lemma9}\cite{FDM} Any irreducible weight $\md_1$-module is isomorphic to one of the following
$$t_1^{\lambda_1}\C[t_1^{\pm 1}],\ \lambda_1\not\in\Z, \ \C[t_1],\ \C[t_1^{\pm 1}]/\C[t_1].$$
\end{lemma}

\begin{theorem}\label{th9}
Let $M$ be an irreducible Harish-Chandra
$\ms_n$-module with a nonzero central charge, say $\dot{z}\in \C^*$.
\begin{enumerate}[$($a$)$]
\item If $e$ acts locally nilpotently on $M$, then
$M$ is an irreducible  highest weight module. Explicitly
$$M \cong V\otimes_{\z} L_{\sl_2}(\lambda) \otimes _{\dot z} \C[t_1,\dots,t_n]$$ for some finite dimensional  irreducible
$\so_{n}$-module $V$.
\item If $f$ acts locally nilpotently on $M$, then $M$ is an irreducible lowest weight module. Explicitly
$$M^\tau \cong V\otimes_{\z} L_{\sl_2}(\lambda) \otimes _{\dot z} \C[t_1,\dots,t_n]$$ for some finite dimensional  irreducible
$\so_{n}$-module $V$, where $\tau$ is an involution of $\ms_n$ such that $\tau(e)=-f, \tau(h)=-h, \tau(x_i)=-y_i,\tau(s_{ij})=s_{ij}$, for $ 1\leq i,j\leq n$.
\item  If both $e$ and $f$ act injectively on $M$, then $n=1$ and
\begin{equation}\label{3.5}M\cong L_{\sl_2}(k)\otimes_{\z} t_1^\lambda\C[t_1,t_1^{-1}] ,\end{equation}
where $k\in\Z_+, \lambda\not\in\Z$.
\end{enumerate}
\end{theorem}

\begin{proof} (a). By Lemma \ref{lemma-nil}, $M$ is an irreducible  highest weight module. Then (a) follows from
(b) in Proposition \ref{GHM}.

(b). It follows similarly.

(c).  Via the isomorphism $\phi_{\z}$ in Proposition \ref{proposition3}, we can view $M$ as a module over
$U(\so_n\oplus \sl_2)\otimes \md_n$.
 By Lemma \ref{lemma-nil}, $x_i, y_i$ act injectively on $M$, hence $t_i,\partial_i$ act injectively on $M$. Choose a weight vector $v\in M_\lambda$ which is a
 common eigenvector of $t_i\partial_i$, $i\in\{1,\dots,n\}$. If $n>1$, then $\md_n v=\C[t_1^{\pm 1},\dots, t_n^{\pm 1}]$ as vector spaces.
 Note that $1\otimes \tilde{h}$ acts on  $\C[t_1^{\pm 1},\dots, t_n^{\pm 1}]$ by $-\frac{1}{2} \sum_{k=1}^n (\partial_kt_k+t_k\partial_k)$.
  Consequently $M$ has infinite dimensional weight spaces, a contradiction. Hence $n=1$. When $n=1$, the module isomorphism (\ref{3.5}) follows from the algebra
 isomorphism $U(\ms_1)/\langle z-\z \rangle\cong U(\sl_2)\otimes \md_1$ and Lemma \ref{lemma9}.

\end{proof}

Note that in the case $n=1$, the classification was given in \cite{D} by a different method.

\begin{center}
\bf Acknowledgments
\end{center}

\noindent G.L. is partially supported by NSF of China (Grant
11771122).

\vspace{0.2cm}

\noindent G. Liu: School
of Mathematics and Statistics, and  Institute of Contemporary Mathematics, Henan University, Kaifeng 475004, P.R. China. Email:
liugenqiangbnu@126.com
\vspace{0.2cm}

\noindent Y. Li: School
of Mathematics and Statistics, Henan University, Kaifeng 475004, P.R. China. Email:
897981524@qq.com
\vspace{0.2cm}

\noindent K. Wang: School
of Mathematics and Statistics, Henan University, Kaifeng 475004, P.R. China. Email:
2832094678@qq.com

\end{document}